\documentclass[10pt]{amsart}
\usepackage{amsfonts}
\usepackage{ifthen}
\usepackage{amsthm}
\usepackage{amsmath}
\usepackage{graphicx}
\usepackage{amscd,amssymb,amsthm}
\usepackage{enumerate}
\usepackage{epstopdf}
\usepackage{mathrsfs}
\usepackage{hyperref}

\newcommand*\pFqskip{8mu}
\catcode`,\active
\newcommand*\F{\begingroup
        \catcode`\,\active
        \def ,{\mskip\pFqskip\relax}%
        \dopFq
}
\catcode`\,12
\def\dopFq#1#2#3{%
        \left(\genfrac..{0pt}{}{#1}{#2};#3\right)%
        \endgroup
}

\newcounter{minutes}
\divide\time by 60
\newcounter{hours}
\multiply\time by 60 \addtocounter{minutes}{-\time}

\title{Properties of the Tur\'anian of modified Bessel functions}

\author[I. Mez\H o]{Istv\'an Mez\H o${\dag}$}
\address{Department of Mathematics, Nanjing University of Information Science and
Technology, 5 Panxin Rd, Pukou, Nanjing, Jiangsu, P.R. China}
\email{istvanmezo81@gmail.com}

\author[\'A. Baricz]{\'Arp\'ad Baricz${\ddag}$}
\address{Institute of Applied Mathematics, \'Obuda University, 1034 Budapest, Hungary}
\address{Department of Economics,  Babe\c{s}-Bolyai University, Cluj-Napoca 400591, Romania}
\email{bariczocsi@yahoo.com}

\thanks{${\dag}$The research of Istv\'an Mez\H{o} was supported by the Scientific Research Foundation of Nanjing University of Information Science \& Technology, the Startup Foundation for Introducing Talent of NUIST, Project no. S8113062001, and the National Natural Science Foundation for China, Grant no. 11501299.}

\thanks{${\ddag}$The research of \'Arp\'ad Baricz was supported by the J\'anos Bolyai Research Scholarship of the Hungarian Academy of Sciences.}

\newtheorem{theorem}{Theorem}
\newtheorem{corollary}{Corollary}

\begin{document}

\def\thefootnote{}
\footnotetext{ \texttt{File:~\jobname .tex,
          printed: \number\year-\number\month-0\number\day,
          \thehours.\ifnum\theminutes<10{0}\fi\theminutes}
} \makeatletter\def\thefootnote{\@arabic\c@footnote}\makeatother


\maketitle

\begin{abstract}
In this paper some new series and integral representations for the Tur\'anian of modified Bessel functions of the first kind are given,
which give new asymptotic expansions and tight bounds for the Tur\'an determinant in the question. It is shown that in the case of natural
and real order the Tur\'anian can be represented in a relatively compact form, which yields a uniform upper bound for the Tur\'an determinant for
modified Bessel functions of the first kind. Our results complement and improve some of the results from the literature.
\end{abstract}

\section{Introduction}

Modified Bessel functions are among the most important special functions of mathematical physics. They appear in various problems of engineering sciences and mathematical physics. Because of this their properties have been studied from many different point of views by many researchers. Motivated by the rich literature on Tur\'an type inequalities for orthogonal polynomials and special functions the Tur\'an determinant of modified Bessel functions of the first kind have been also studied intensively. For more details we refer to the papers \cite{bpams,bams,bp,bpon,joshi,karp,thiru} and to the references therein. In this paper we focus on the modified Bessel functions of the first kind $I_{\nu}$ and especially on the series and integral representations and asymptotic behavior of the Tur\'anian $$\Delta_{\nu}(x)=I_{\nu}^2(x)-I_{\nu-1}(x)I_{\nu+1}(x).$$
This expression in the recent years appeared in various problems of applied mathematics, biology, chemistry, physics and
engineering sciences, and motivated by these applications some tight bounds were given for $\Delta_{\nu}(x)$ in \cite{baricz}. By using the asymptotic expansion $$I_{\nu}(x)\sim \frac{1}{\sqrt{2\pi\nu}}\left(\frac{ex}{2\nu}\right)^{\nu}$$
as $\nu\to\infty$ through positive real values with $x\neq0$ fixed, we clearly have that
\begin{equation}\label{deltasim}\Delta_{\nu}(x)\sim\frac{1}{2\pi\nu}\left(\frac{ex}{2\nu}\right)^{\nu}
\left(1-\frac{\nu^{2\nu+1}}{(\nu-1)^{\nu-\frac{1}{2}}(\nu+1)^{\nu+\frac{3}{2}}}\right),\end{equation}
which shows that the Tur\'anian of the modified Bessel function of the first kind tends to zero from above as $\nu$ tends to infinity through real values and fixed $x\neq0.$ This is in the agreement with the fact that the Tur\'anian $\Delta_{\nu}(x)$ is decreasing as a function of $\nu,$ see the last part of the third section. In the sequel our aim is to show that it is possible to find two more compact asymptotic expansions for $\Delta_{\nu}(x)$ by using some integral and series representations of the Tur\'anian. The first one is based on a new integral formula of Martin \cite{martin}, while the second one is based on the Neumann integral formula of the product of two modified Bessel functions of the first kind.

\section{New representations of the Tur\'anian $\Delta_{\nu}(x)$ for $\nu$ natural number}

\subsection{Series representation of the Tur\'anian $\Delta_n(x)$} In this section we are interested on the asymptotic behavior of $\Delta_n(x)$ as $x$ is fixed and $n$ tends to infinity. Our aim is to prove that for any fixed real $x$ and non-negative integer $n$
\begin{equation}\label{asymp}\Delta_n(x)\sim\frac{x^{2n}}{n!(n+1)!}\frac{1}{2^{2n}}\quad\mbox{as}\quad n\to\infty\end{equation}
holds. In fact, we give a full asymptotic description in the form of a convergent series in the variable $n$. In order to prove this result we need the following integral and series representation for $\Delta_n(x),$ from which \eqref{asymp} immediately follows.

\begin{theorem}\label{th1}
Let $n$ be a non-negative integer, and $x$ be a real number. Then
\begin{equation}\label{approx}\Delta_n(x)=\frac{(-1)^{n}}{\pi x}\int_{-\frac{\pi}{2}}^{\frac{\pi}{2}}\frac{I_1(2x\sin t)}{\sin t}\cos(2n t)dt=\sum_{m\geq n}C_{2m}^{m-n}\frac{\left(\frac{1}{2}x\right)^{2m}}{m!(m+1)!}.\end{equation}
\end{theorem}

\begin{proof} By definition,
\begin{equation}
I_n(x)=i^{-n}J_n(ix),\label{IJrel}
\end{equation}
where $J_n$ is the Bessel function of the first kind. From here we get that
\[\Delta_n(x)=(-1)^{-n}\left(J_n^2(ix)-J_{n-1}(ix)J_{n+1}(ix)\right).\]
By a result of Martin \cite[eq. (21)]{martin} the right-hand side of the above expression can be expressed as a Fourier integral for integer $n$, that is,
\[J_n^2(x)-J_{n-1}(x)J_{n+1}(x)=\frac{1}{2\pi}\int_{-\pi}^\pi\frac{J_1(2x\sin(t/2))}{x\sin(t/2)}e^{in t}dt.\]
Replacing $x$ by $ix$ on both sides and recalling \eqref{IJrel}, a transformation of variables ($t\to t/2$) yields the first result in \eqref{approx}. Note that we also separated the real part, hence the exponential turns to be cosine in the statement of the theorem.

Now, let us focus on the second equality in \eqref{approx}. We apply the integral representation
\[\Delta_n(x)=\frac{(-1)^{n}}{\pi x}\int_{-\frac{\pi}{2}}^{\frac{\pi}{2}}\frac{I_1(2x\sin t)}{\sin t}\cos(2n t)dt\]
together with the MacLaurin expansion of $I_1$
\[I_1(2x\sin t)=\sum_{m\geq0}\frac{\left(x\sin t\right)^{2m+1}}{m!(m+1)!}.\]
Substituting this into the integral formula, we get that
\begin{align*}\Delta_n(x)&=\frac{(-1)^{n}}{\pi x}\sum_{m\geq0}\frac{x^{2m+1}}{m!(m+1)!}\int_{-\frac{\pi}{2}}^{\frac{\pi}{2}}\frac{\sin^{2m+1}(t)\cos(2n t)}{\sin t}dt\\
&=\sum_{m\geq0}\frac{x^{2m}}{m!(m+1)!}\frac{(-1)^{n}}{\pi}\int_{-\frac{\pi}{2}}^{\frac{\pi}{2}}\sin^{2m}(t)\cos(2n t)dt=\sum_{m\geq0}\frac{x^{2m}}{m!(m+1)!}T_n(m),\end{align*}
where
\begin{align*}
T_n(m)&=\frac{(-1)^n}{\pi}\int_{-\frac{\pi}{2}}^{\frac{\pi}{2}}\sin^{2m}(t)\cos(2n t)dt\\
&=\frac{(-1)^n}{\pi}\int_{0}^{\pi}\cos^{2m}(u)\cos(2n u)\cos(n\pi)du\\
&=\frac{2}{\pi}\int_{0}^{\frac{\pi}{2}}\cos^{2m}(u)\cos(2n u)du\\
&=\frac{1}{2^{2m}}\frac{\left(B(m+n+1,m-n+1)\right)^{-1}}{2m+1}=\frac{1}{2^{2m}}C_{2m}^{m-n},
\end{align*}
according to \cite[eq. 5.12.5]{nist} and by using the fact that the graph of the function $u\mapsto \cos^{2m}u\cos(2nu)$ is symmetric with respect to the vertical line $x=\pi/2.$ Note that $T_n(m)$ vanishes when $m<n$. The first two non-vanishing values are
\[T_n(n)=\frac{1}{2^{2n}}\quad \mbox{and}\quad T_n(n+1)=\frac{n+1}{2^{2n+1}},\]
and the proof of the asymptotic expansion \eqref{asymp} is done, taking the first term approximation in \eqref{approx}.
\end{proof}

\subsection{Bounds for the Tur\'anian $\Delta_n(x)$} We know that for $\nu>-1$ and $x>0$ we have $\Delta_\nu(x)\ge0,$ see \cite{bpams,bams,baricz,bpon,joshi,karp,thiru} for more details. A better lower bound easily comes at least for non-negative integer $\nu=n$ after realizing that all the $T_n(m)$ integrals above are positive in \eqref{approx}. Hence for $n\in\mathbb{N}$ and $x>0$ we have
\begin{equation}\label{turan1}\Delta_n(x)\ge\frac{1}{4^n}\frac{x^{2n}}{n!(n+1)!}\end{equation}
as well as the even sharper two term cut of \eqref{approx}
\begin{equation}\label{turan2}\Delta_n(x)\ge\frac{1}{4^n}\frac{x^{2n}}{n!(n+1)!}+\frac{1}{2^{2n+1}}\frac{x^{2(n+1)}}{n!(n+2)!}.\end{equation}
It is clear that it is possible to have more sharper lower bound for $\Delta_n(x)$ by taking more terms of the series on the right-hand side of \eqref{approx}. We also mention that the Tur\'an type inequality \eqref{turan1} is not new, it was proved first by Kalmykov and Karp \cite[eq. (26)]{karp}, and it is actually
valid in the following form
\begin{equation}\label{turan3}\Delta_{\nu}(x)\ge\frac{1}{4^\nu}\frac{x^{2\nu}}{(\nu+1)\Gamma^2(\nu+1)},\end{equation}
where $\nu>-1$ and $x>0.$ Taking into account this it is natural to ask whether the Tur\'an type inequality \eqref{turan2} and the relations in \eqref{approx}
are valid for arbitrary real parameter $\nu>-1.$ An affirmative answer to this question will be given in section 2.

Now, our goal is to find an upper estimation by using our series representation in \eqref{approx}.

\begin{theorem}
For $n$ positive integer we have
\begin{equation}
\Delta_n(x)\le \rho_n\frac{x^{2n}}{n!(n+1)!}{}_1F_2(1;1+n,2+n;x^2),\label{Festim}
\end{equation}
where $$\rho_n=\sup\limits_{m\geq n}T_n(m)=\sup_{m\geq n}\frac{1}{2^{2m}}C_{2m}^{m-n}=\frac{1}{2^{4n^2-2}}\frac{(4n^2-2)!}{(2n^2+n-1)!(2n^2-n-1)!}.$$
\end{theorem}

\begin{proof}
We make use of our series representation \eqref{approx} for $\Delta_n(x)$
\[\Delta_n(x)=\sum_{m\geq n}\frac{x^{2m}}{m!(m+1)!}T_n(m)\le \rho_n\sum_{m\geq n}\frac{x^{2m}}{m!(m+1)!}=\rho_n\frac{x^{2n}}{n!(n+1)!}{}_1F_2(1;1+n,2+n;x^2).\]
To finish the proof we show that $T_n(m)$ for a fixed $n$ grows to a certain $m$ and then it steadily decreases. Thus, $\rho_n$ is actually a maximum with a given finite $n$. To prove this we show that $\left\{T_n(m)\right\}_{m\geq n}$ is a unimodal sequence. Indeed, since
$$\frac{T_n(m+1)}{T_n(m)}=\frac{1}{2}\frac{(2m+1)(m+1)}{(m+n+1)(m-n+1)}\gtrless1$$ if and only if $m\lessgtr2n^2-1,$ it follows that the sequence $\left\{T_n(m)\right\}_{m\geq n}$ and it takes its maximum at $m=2n^2-1.$
\end{proof}

\subsection{Some remarks on the numerical performance of the bounds} Since all of our bounds in the previous subsection were deduced from the still exact \eqref{approx}, a short check reveals that our bounds are getting sharper as $x$ is smaller or $n$ is bigger. It is also worth to compare \eqref{Festim} with existing results. A known upper bound for $\Delta_n(x)$ says that for $n\in\mathbb{N}$ and $x\in\mathbb{R}$ we have \cite{baricz,joshi,thiru}
\begin{equation}
\Delta_n(x)\le\frac{1}{n+1}I_n^2(x).\label{Arpi}
\end{equation}
Our numerical simulations show that this inequality is sharper than \eqref{Festim} for small $x$, however, for moderate $x$ and bigger, \eqref{Festim} starts to perform better than \eqref{Arpi}. Then, as $n$ grows, the right hand side of \eqref{Arpi} starts to be closer to $\Delta_n(x)$ than our estimation.

\subsection{Another bound for $\Delta_n(x)$}

In principle, knowing the $m$-dependent asymptotic behavior of $T_n(m)$ could give a better approximation than that of a simple use of $\rho_n=\sup_mT_n(m)$ as we did in the last subsections. It turns out, however, that $T_n(m)$, as $m$ grows, has a so weak dependence of the variable $n$ than in the first order approximation it does not appear at all. Namely, we are first going to prove that

\begin{theorem}\label{TasympProp}We have that
\begin{equation}
T_n(m)\sim\frac{1}{\sqrt{\pi}}\frac{1}{\sqrt{m}}\quad\mbox{as}\quad m\to\infty.\label{Tasymp}
\end{equation}
\end{theorem}

From here it follows that $T_n(m)\le c\frac{1}{\sqrt{m}}$ uniformly for all $m$ with some constant $c,$ independent from $n$. Hence we can deduce the following.
\begin{corollary}For any positive integer $n$
\begin{equation}
\Delta_n(x)\le c\sum_{m\geq n}\frac{x^{2m}}{m!(m+1)!}\frac{1}{\sqrt{m}}.\label{Delta_and_c}
\end{equation}
\end{corollary}

\begin{proof}(of Theorem \ref{TasympProp}.)
The generating function of $T_n(m)$ is
\[\sum_{m\geq0} T_n(m)x^m=\frac{(-1)^n}{\pi}\int_{-\frac{\pi}{2}}^{\frac{\pi}{2}}\cos(2n t)\sum_{m\geq0}\sin^{2m}(t)x^mdt=\frac{(-1)^n}{\pi}\int_{-\frac{\pi}{2}}^{\frac{\pi}{2}}\frac{\cos(2n t)}{1-x\sin^2t}.\]
This integral can be evaluated in terms of hypergeometric functions. It results that
\begin{equation}
\sum_{m\geq0} T_n(m)x^m=\frac{(-1)^n}{1-x}\frac{\sin(\pi n)}{\pi n}{}_3F_2\F{1/2,1,1}{1-n,1+n}{\frac{x}{x-1}}.\label{hypergeom}
\end{equation}
Although the sine factor is singular for integer $n$ the whole expression is not. What is important for us is the pole structure of the right-hand side. If we know this, we can use the standard Darboux method. It can be seen that the right-hand side is singular at $x=1$, and there are no poles of smaller modulus, just note that the integral
\[\int_{-\frac{\pi}{2}}^{\frac{\pi}{2}}\frac{\cos(2n t)}{1-x\sin^2t}\]
is convergent for $|x|<1$. Analyzing the series expansion of the right hand side of \eqref{hypergeom} at the pole $x=1$ it comes that
\[\frac{(-1)^n}{1-x}\frac{\sin(\pi n)}{\pi n}{}_3F_2\F{1/2,1,1}{1-n,1+n}{\frac{x}{x-1}}=\frac{1}{\sqrt{1-x}}-2n-2n^2\sqrt{1-x}+\mathcal{O}(1-x).\]
This observation revels that $x=1$ is an algebraic singularity, and $T_n(m)$ behaves asymptotically like the coefficients in $\frac{1}{\sqrt{1-x}}$. The latter is known as
\[[x^m]\frac{1}{\sqrt{1-x}}=\frac{1}{\sqrt\pi}\frac{1}{\sqrt m}.\]
Our result follows as we stated.
\end{proof}

The constant $c$ in \eqref{Delta_and_c} can explicitly be determined as the following statement shows.

\begin{theorem}For any positive integer $n$ and real $x$
$$\Delta_n(x)\le \frac{1}{\sqrt{\pi}}\sum_{m\geq n}\frac{x^{2m}}{m!(m+1)!}\frac{1}{\sqrt{m}}.$$
\end{theorem}

\begin{proof}From the definition of the constant $c$ (see after \eqref{Tasymp}) we have that it can be chosen as
\[c=\sup_{m\ge n}\sqrt{m}\cdot T_n(m).\]
Therefore it is enough to find this supremum (as we will see, it is independent from $n$). We are going to show that the sequence $\left\{\sqrt{m}T_n(m)\right\}_{m\geq n}$ is strictly increasing, so the supremum will be in fact the limit
\[c=\lim_{m\ge n}\sqrt{m}\cdot T_n(m).\]
As we saw earlier, $T_n(m)=2^{-2m}C_{2m}^{m-n}$, from where we get that
\[\frac{\sqrt{m+1}T_n(m+1)}{\sqrt{m}T_n(m)}=\frac{(1+m)^{3/2}(2m+1)}{2\sqrt{m}(1+m-n)(1+m+n)}.\]
In order to prove that the sequence $\left\{\sqrt{m}T_n(m)\right\}_{m\geq n}$ is strictly increasing, we must show that the right hand side is not less than 1. That is, we need to prove that
\[\frac{(1+m)^{3/2}(2m+1)}{2\sqrt{m}}\ge(1+m)^2-n^2.\]
The stronger inequality
\[\frac{(1+m)^{3/2}(2m+1)}{2\sqrt{m}}\ge(1+m)^2\]
can also be proven in an elementary way. That $\sqrt{m}T_n(m)\to\frac{1}{\sqrt{\pi}}$ as $m\to\infty$ comes from Theorem \ref{TasympProp}, and hence the proof is complete.
\end{proof}

\section{Series representation of the Tur\'anian $\Delta_{\nu}(x)$ for $\nu$ real number}
\setcounter{equation}{0}

\subsection{Integral representation of $\Delta_{\nu}(x)$} Now, we are going to show the general result of \eqref{approx} in Theorem \ref{th1}.

\begin{theorem}\label{th3}
For $\nu>-\frac{1}{2}$ and $x\in\mathbb{R}$ we have the following representations
\begin{equation}\label{serint}\Delta_{\nu}(x)=\frac{4}{\pi}\int_0^{\frac{\pi}{2}}I_{2\nu}(2x\cos\theta)\sin^2\theta d\theta=\frac{1}{\sqrt{\pi}}\sum_{m\geq0}\frac{\Gamma\left(\nu+m+\frac{1}{2}\right)x^{2\nu+2m}}{m!\Gamma(\nu+m+2)\Gamma(2\nu+m+1)}.\end{equation}
Moreover, the series representation is valid for all $\nu>-1$ and $x\in\mathbb{R}.$
\end{theorem}

\begin{proof}
In view of the Neumann formula \cite[eq. 10.32.15]{nist}
$$I_{\mu}(x)I_{\nu}(x)=\frac{2}{\pi}\int_0^{\frac{\pi}{2}}I_{\mu+\nu}(2x\cos\theta)\cos((\mu-\nu)\theta)d\theta, \quad \mu+\nu>-1,$$
we have for $\nu>-\frac{1}{2}$ and $x\in\mathbb{R}$
\begin{align*}
\Delta_{\nu}(x)&=\frac{2}{\pi}\int_0^{\frac{\pi}{2}}I_{2\nu}(2x\cos\theta)d\theta-\frac{2}{\pi}\int_0^{\frac{\pi}{2}}I_{2\nu}(2x\cos\theta)\cos(2\theta)d\theta\\
&=\frac{4}{\pi}\int_0^{\frac{\pi}{2}}I_{2\nu}(2x\cos\theta)\sin^2\theta d\theta\\
&=\frac{4}{\pi}\int_0^{\frac{\pi}{2}}\sin^2\theta \sum_{m\geq0}\frac{(x\cos\theta)^{2m+2\nu}}{m!\Gamma(m+2\nu+1)}d\theta\\
&=\frac{4}{\pi}\sum_{m\geq0}\frac{x^{2m+2\nu}}{m!\Gamma(m+2\nu+1)}\int_0^{\frac{\pi}{2}}\sin^2\theta \cos^{2m+2\nu}\theta d\theta\\
&=\frac{2}{\pi}\sum_{m\geq0}\frac{B\left(\frac{3}{2},\nu+m+\frac{1}{2}\right)x^{2m+2\nu}}{m!\Gamma(m+2\nu+1)}\\
&=\frac{1}{\sqrt{\pi}}\sum_{m\geq0}\frac{\Gamma\left(\nu+m+\frac{1}{2}\right)x^{2m+2\nu}}{m!\Gamma(\nu+m+2)\Gamma(2\nu+m+1)}.
\end{align*}
The series representation can be obtained also by using the Cauchy product \cite[eq. 10.31.3]{nist}
$$I_{\nu}(x)I_{\mu}(x)=\left(\frac{x}{2}\right)^{\nu+\mu}\sum_{m\geq0}\frac{(\nu+\mu+m+1)_m\left(\frac{x}{2}\right)^{2m}}{m!\Gamma(\nu+m+1)\Gamma(\mu+m+1)}.$$
Namely, by using this formula and the Legendre duplication formula for the gamma function we get
\begin{align*}\Delta_{\nu}(x)=\sum_{m\geq0}\frac{(2\nu+m+1)_m}{m!}\frac{\left(\frac{x}{2}\right)^{2\nu+2m}}{\Gamma(\nu+m+1)\Gamma(\nu+m+2)}
=\frac{1}{\sqrt{\pi}}\sum_{m\geq0}\frac{\Gamma\left(\nu+m+\frac{1}{2}\right)x^{2\nu+2m}}{m!\Gamma(\nu+m+2)\Gamma(2\nu+m+1)}.\end{align*}
\end{proof}

\subsection{Bounds for $\Delta_{\nu}(x)$} From this result in view of the Legendre duplication formula \cite[eq. 5.5.5]{nist} $$2\sqrt{\pi}\Gamma(2x)=4^x\Gamma(x)\Gamma\left(x+\frac{1}{2}\right)$$ and the recurrence relation $\Gamma(x+1)=x\Gamma(x),$ it is clear that for $\nu>-1$ and $x\in\mathbb{R}$ we have
\begin{equation}\label{turan4}
\Delta_{\nu}(x)>\frac{1}{\sqrt{\pi}}\frac{\Gamma\left(\nu+\frac{1}{2}\right)x^{2\nu}}{\Gamma(\nu+2)\Gamma(2\nu+1)}=
\frac{\left(\frac{1}{4}x^{2}\right)^{\nu}}{\Gamma(\nu+1)\Gamma(\nu+2)}
\end{equation}
and
\begin{equation}\label{turan5}
\Delta_{\nu}(x)>\frac{1}{\sqrt{\pi}}\frac{\Gamma\left(\nu+\frac{1}{2}\right)x^{2\nu}}{\Gamma(\nu+2)\Gamma(2\nu+1)}+
\frac{1}{\sqrt{\pi}}\frac{\Gamma\left(\nu+\frac{3}{2}\right)x^{2\nu+2}}{\Gamma(\nu+3)\Gamma(2\nu+2)}=
\frac{\left(\frac{1}{4}x^{2}\right)^{\nu}}{\Gamma(\nu+1)\Gamma(\nu+2)}+\frac{2\left(\frac{x}{2}\right)^{2\nu+2}}{\Gamma(\nu+1)\Gamma(\nu+3)},
\end{equation}
and clearly we can have more sharper lower bounds for $\Delta_{\nu}(x)$ by taking more terms of the series in \eqref{serint}. Observe that
the Tur\'an type inequality \eqref{turan4} coincides with the inequality \eqref{turan3} of Kalmykov and Karp, while the Tur\'an type inequality \eqref{turan5}
is the extension of \eqref{turan2} to real variable $\nu.$

\subsection{Some remarks on $\Delta_{\nu}(x)$} From Theorem \ref{th1} and its proof yields that the Tur\'anian of the Bessel function $J_n$ has
a similar structure as the Tur\'anian $\Delta_n(x)$ in Theorem \ref{th1}, with the difference that the series of the Tur\'anian of $J_n$ has alternating coefficients. Hence it is immediate that for all $n$ natural number and real $x$ we have
$$J_n^2(x)-J_{n-1}(x)J_{n+1}(x) < I_n^2(x)-I_{n-1}(x)I_{n+1}(x).$$
Moreover, it is possible to prove the counterpart of this inequality as
$$J_n^2(x)-J_{n-1}(x)J_{n+1}(x)> \frac{J_n^2(x)}{I_n^2(x)}\left(I_n^2(x)-I_{n-1}(x)I_{n+1}(x)\right),$$
and even for real order. Namely, if $j_{\nu,n}$ stands for the $n$th positive zero of the Bessel function of the first kind $J_{\nu},$ then by using the relations (see \cite{bams,baricz,bp,joshi,skov} for more details)
$$\frac{1}{4I_{\nu}^2(x)}\left(I_{\nu}^2(x)-I_{\nu-1}(x)I_{\nu+1}(x)\right)=\sum_{n\geq 1}\frac{j_{\nu,n}^2}{(x^2+j_{\nu,n}^2)^2}$$
and
$$\frac{1}{4J_{\nu}^2(x)}\left(J_{\nu}^2(x)-J_{\nu-1}(x)J_{\nu+1}(x)\right)=\sum_{n\geq 1}\frac{j_{\nu,n}^2}{(x^2-j_{\nu,n}^2)^2},$$
we obtain that for $\nu>-1$ and $x\neq0$ the next Tur\'an type inequality is valid
$$J_{\nu}^2(x)-J_{\nu-1}(x)J_{\nu+1}(x)> \frac{J_{\nu}^2(x)}{I_{\nu}^2(x)}\left(I_{\nu}^2(x)-I_{\nu-1}(x)I_{\nu+1}(x)\right).$$

\subsection{Monotonicity of $\Delta_{\nu}(x)$ with respect to $\nu$} Since the function $x\mapsto x^{x+\frac{1}{2}}$ is log-convex on $(0,\infty)$ for
$\nu>1$ we have the inequality $\nu^{2\nu+1}<(\nu-1)^{\nu-\frac{1}{2}}(\nu+1)^{\nu+\frac{3}{2}},$ and hence the right-hand side of the asymptotic expansion \eqref{deltasim} is positive, and it is actually decreasing for large $\nu.$ This suggest that $\nu\mapsto\Delta_{\nu}(x)$ has a similar behavior. Indeed, owing to Watson \cite{swatson} we know that the function $x\mapsto I_{\nu+1}(x)/I_{\nu}(x)$ is increasing on $(0,\infty)$ for $\nu\geq -\frac{1}{2},$ and consequently, the function $x\mapsto I_{\nu-1}(x)/I_{\nu}(x)$ is decreasing on $(0,\infty)$ for $\nu\geq \frac{1}{2}.$ By using this and the recurrence relation $2I_{\nu}'(x)=I_{\nu-1}(x)+I_{\nu+1}(x)$ twice we obtain
$$\left(\frac{I_{\nu-1}(x)}{I_{\nu}(x)}\right)'=\frac{I_{\nu-1}'(x)I_{\nu}(x)-I_{\nu-1}(x)I_{\nu}'(x)}{I_{\nu}^2(x)}=
\frac{\Delta_{\nu}(x)-\Delta_{\nu-1}(x)}{2I_{\nu}^2(x)}\leq0$$ for $x>0$ and $\nu\geq \frac{1}{2}.$ This, in particular implies that
the sequence $\{\Delta_n(x)\}_{n\geq0}$ is decreasing for fixed $x$ real number. Moreover, since the function $\nu\mapsto I_{\nu}(x)$ is decreasing on $[0,\infty)$ for each $x>0$ fixed, according to Cochran \cite{co}, by using the first part of \eqref{serint} we obtain that $\nu\mapsto \Delta_{\nu}(x)$ is
decreasing on $[0,\infty)$ for $x>0$ fixed.

\end{document}